\documentclass{amsart}

\usepackage{Mytheorems}

\usepackage{amssymb, amsmath}
\usepackage{amsfonts}
\usepackage{blindtext}
\usepackage{color}
\usepackage{comment}
\usepackage{epsfig}
\usepackage{float}
\usepackage{graphicx}
\usepackage{array,makecell}
\newtheorem*{thmA*}{Theorem A}
\newtheorem*{thmB*}{Theorem B}
\newtheorem*{lemA*}{Lemma A}

\newtheorem*{assumptionA*}{Assumption A}
\newtheorem*{assumptionB*}{Assumption B}
\newtheorem*{assumptionC*}{Assumption C}
\newtheorem{Winfree Model}{Winfree Model}

\newcommand{\R}{\mathbb R}

\def\be{\begin{equation}}
\def\ee{\end{equation}}

\def\R{\mathbb R}

\def\tilde{\widetilde}
\renewcommand{\theequation}{\thesection.\arabic{equation}}
\numberwithin{figure}{section}
\vspace{1cm}

\def\R{\mathbb{R}}
\def\O{\Omega}
\DeclareMathOperator{\diam}{{\rm diam}}

\makeatletter
\newcommand{\opnorm}{\@ifstar\@opnorms\@opnorm}
\newcommand{\@opnorms}[1]{%
  \left|\mkern-1.5mu\left|\mkern-1.5mu\left|
   #1
  \right|\mkern-1.5mu\right|\mkern-1.5mu\right|
}
\newcommand{\@opnorm}[2][]{%
  \mathopen{#1|\mkern-1.5mu#1|\mkern-1.5mu#1|}
  #2
  \mathclose{#1|\mkern-1.5mu#1|\mkern-1.5mu#1|}
}
\makeatother

\allowdisplaybreaks

\begin{document}
\bibliographystyle{siam}

\title[]
{On the Existence of $H^1$ solutions for Stationary Linearized Boltzmann Equations in a Small Convex Domain}

\author[I-K.~Chen, P.-H.~Chuang, C.-H.~Hsia, D.~Kawagoe and J.-K.~Su]{I-KUN CHEN, PING-HAN CHUANG, CHUN-HSIUNG HSIA, DAISUKE KAWAGOE AND JHE-KUAN SU}

\date{\today}

\begin{abstract}

In this article, we investigate the incoming boundary value problem for the stationary linearized Boltzmann equations in $ \Omega \subseteq \mathbb{R}^{3}$. For a $C^2$ bounded domain with boundary of positive Gaussian curvature, the existence theory is established in $H^{1}(\Omega \times \mathbb{R}^{3})$ provided that the diameter of the domain $\Omega$ is small enough.    \end{abstract}

\maketitle

\section{Introduction} 


In this article, we are interested in the existence and regularity theory of the following stationary linearized Boltzmann equation: 
\begin{equation}\label{stationarylinearized Boltzmann equation}
v \cdot \nabla_{x}f(x,v)=L(f)(x,v) \mbox{ for } (x,v) \in \Omega \times \mathbb{R}^{3}.
\end{equation}
Here, $\Omega$ is a domain in $\R^3$ with $C^2$ boundary $\partial \Omega$ and $L$ is the linearized collision operator.
The existence of solutions of lower regularity has been studied for linearized and full equations with various boundary conditions, see \cite{Gui 1,Guiraudlinear,Gui 2,GuoKim}. 
As for regularity issue, based on existence results above, the pointwise regularity is studied in \cite{RegularChen,CHK,chenkim,chenH}.  
A recent result \cite{I kun 1} studies the regularity in fractional Sobolev spaces. It proves that solutions are in $L^2(\mathbb{R}^3, H^s(\Omega))$ for $0 \leq s < 1$. Furthermore, the $L^2(\mathbb{R}^3, H^1(\Omega))$ regularity cannot be achieved by the estimates provided in \cite{I kun 1}. It is natural to make a comparison with a famous regularity result \cite{GKTT}, in which the $W^{1,p}$ regularity is proved for the time evolutional full Boltzmann equation with the diffuse boundary condition for $1 \leq p < 2$ and a counterexample for $H^1$ solution is provided for the free transport equation. However, one should be aware that the time evolutional problem does not provide much information for the stationary problem in regularity issue. For example, the Laplace equation can be regarded as the stationary equation of the wave equation while they behave very differently in regard of regulaity. This motivates us to study whether $H^1$ is the critical function space for regularity of the stationary Boltzmann equation. With the help of the smallness of the domain, we prove existence of the $H^1$ solution for the stationary linearized Boltzmann equation with the incoming boundary condition, which gives a partial answer. 

We introduce the setup of problem as follows. We assume that $L(f)$ satisfies the following assumption.
\begin{assumptionA*}
The operator $L(f)$ can be decomposed into the multiplicative term $-\nu(v)f(x,v)$ and the integral operator term
\begin{equation} \label{K}
K(f)(x,v):=\int_{\mathbb{R}^{3}}k(v,v^{*})f(x,v^{*})dv^{*} 
\end{equation}
such that
\begin{align}&k(v,v^*)=k(v^*,v),\\
&\nu_{0}(1+\vert v \vert)^{\gamma} \leq \nu(v) \leq \nu_{1} (1+\vert v \vert)^{\gamma}, \label{AA}\\
&\vert k(v,v^{*}) \vert \lesssim \frac{1}{\vert v-v^{*} \vert(1+\vert v \vert +\vert v^{*} \vert)^{1-\gamma}}e^{-\frac{1-\rho}{4}(\vert v-v^{*} \vert^{2}+(\frac{\vert v \vert^{2}-\vert v^{*} \vert^{2}}{\vert v-v^{*} \vert})^{2})}, \label{AB}\\
&\vert \nabla_{v}k(v,v^{*}) \vert \lesssim \frac{1+\vert v \vert}{\vert v-v^{*} \vert^{2}(1+\vert v \vert +\vert v^{*} \vert)^{1-\gamma}}e^{-\frac{1-\rho}{4}(\vert v-v^{*} \vert^{2}+(\frac{\vert v \vert^{2}-\vert v^{*} \vert^{2}}{\vert v-v^{*} \vert})^{2})}, \label{AC}\\
&|\nabla_{v}\nu(v)|\lesssim(1+|v|)^{\gamma-1} \label{AD},
\end{align}
where $\rho \in (0,1)$ and $\gamma \in [0,1]$.
\end{assumptionA*}
Here we adopt the convention $f \lesssim g$ if there exists a positive constant $C$ such that $f \leq Cg$.

Meanwhile, the problem \eqref{stationarylinearized Boltzmann equation} is supplemented with the incoming boundary condition: 
\begin{equation}
\label{inbdry}
f(x,v)=g(x,v) \mbox{ for } (x,v) \in \Gamma^{-},
\end{equation}
where \begin{equation}\label{Gamma-def}\Gamma^{-}:=\lbrace (x,v)\in \partial \Omega \times \mathbb{R}^{3} \mid n(x)\cdot v<0 \rbrace\end{equation} and $n(x)$ is the outward unit normal of $\partial\Omega$ at $x$.

\begin{remark}

\begin{enumerate}
\item If we adopt the idea of Grad \cite{Grad} and consider the Grad angular cut-off potentials which include the hard sphere, hard potential, and Maxwellian molecular condition, then the condition of \eqref{AB} and the upper bound of \eqref{AA} hold. See Caflisch \cite{Caf 1}.

\item It is worth mentioning that the commonly used cross section $B(|v-v^*|,\theta)= b|v-v^*|^{\gamma} \cos \theta$, where $b$ is a positive constant, leads to all the estimates in Assumption A.
\end{enumerate}
\end{remark}

For the convenience of further discussion, we define
\begin{align*}
\tau_{x,v} &:= \inf\lbrace s > 0 \mid x-sv\in\Omega^{c} \rbrace,\\
q(x,v) &:= x-\tau_{x,v}v.
\end{align*}
With this notation, the equation \eqref{stationarylinearized Boltzmann equation} can be expressed by the following integral form:
\begin{equation} \label{integral form}
\begin{split}
f =& e^{-\nu(v)\tau_{x,v}}g(q(x,v),v)+\int_{0}^{\tau_{x,v}}e^{-\nu(v)s}Kf(x-sv,v)ds\\
=&Jg+S_{\Omega}Kf,
\end{split}
\end{equation}
where
\begin{align}
Jg(x, v) :=& e^{-\nu(v)\tau_{x,v}}g(q(x,v),v), \label{J}\\   
S_{\Omega}h(x, v) :=& \int_{0}^{\tau_{x,v}}e^{-\nu(v)s}h(x-sv,v)ds. \label{S}
\end{align}
Notice that we say $f$ is a solution to \eqref{stationarylinearized Boltzmann equation} with boundary condition \eqref{inbdry} if $f$  satisfies \eqref{integral form}.

We assume that the domain $\Omega$ possesses the following property.
\begin{assumptionB*}
$\Omega \subset \mathbb R^3$ is a $C^2$ bounded domain such that $\partial \Omega$ is of positive Gaussian curvature.
\end{assumptionB*}

Our first main result is the following theorem.
\begin{theorem} \label{main theorem 1}
Suppose $L$ satisfies {\bf Assumption A}, then there exists $\epsilon >0$ such that: for any domain $\Omega$ satisfying {\bf Assumption B} with $\diam(\Omega)<\epsilon$, the boundary value problem \eqref{stationarylinearized Boltzmann equation}-\eqref{inbdry} has a  unique solution $f \in H^{1}( \Omega \times \mathbb R^3)$ if and only if $Jg\in H^{1}(\Omega \times \mathbb{R}^{3})$.
\end{theorem}

We remark that the condition $Jg \in H^{1}$ in the statement of Theorem \ref{main theorem 1} is implicit. To demonstrate that there is a wide class of functions satisfying this condition, we have the following lemma.

\begin{lemma} \label{lem:H1_Jg}
Let $g:\Gamma^{-} \to \mathbb{R}.$ Suppose $g$ satisfies the following conditions.
\begin{align}
\int_{\mathbb{R}^{3}}\int_{\Gamma_v^-}\vert g(z,v)\vert^{2}d\Sigma(z)dv <& \infty, \label{REMARK est 1}\\
\int_{\mathbb{R}^{3}}\frac{\nu(v)^{2}}{\vert v \vert ^{2}}\int_{\Gamma_v^-}\vert g(z,v)\vert^{2} d\Sigma(z)dv <& \infty, \label{REMARK est 2}\\
\int_{\mathbb{R}^{3}}\int_{\Gamma_{v}^{-}}\vert \nabla_{x}g(z,v)\vert^{2} d\Sigma(z)dv <& \infty, \label{REMARK est 3}\\
\int_{\mathbb{R}^{3}}\int_{\Gamma_{v}^{-}}\vert \nabla_{v}g(z,v)\vert^{2}N^{2}(z,v) d\Sigma(z)dv <& \infty, \label{REMARK est 4}
\end{align}
where
\begin{align*}
\Gamma_{v}^{-} &:= \lbrace x \in \partial \Omega \mid n(x)\cdot v <0 \rbrace, \\   
N(x,v) &:= -n(q(x,v)) \cdot \frac{v}{\vert v \vert}, \\ 
|\nabla_{x}h(z,v)|^{2} &:= g^{ij}\nabla_{\frac{\partial}{\partial x_{i}}}h\nabla_{\frac{\partial}{\partial x_{j}}}h,
\end{align*}
$d\Sigma(z)$ is the surface measure of $\partial \Omega$ and $g^{ij}$ is the $(i,j)$ element of the inverse matrix of the metric tensor on $\partial \Omega$. Then we have $Jg \in H^{1}(\Omega \times \mathbb{R}^{3})$.
\end{lemma}
The detailed definition of $|\nabla_{x}h(z,v)|^{2}$ and the proof is presented in the subsection.
 It can be verified by the following standard scaling analysis:
\begin{equation}
v\cdot \nabla_{x}f =\frac{1}{\kappa}(-\nu(v)f+K(f)), \quad (x,v)\in \Omega'\times \mathbb{R}^{3}.
\end{equation}

Notice that by defining $f_\kappa(x,v):=f(\kappa x,v)$, we have 
\begin{equation}
v\cdot \nabla_{x}f_\kappa =-\nu(v)f_\kappa + K(f_\kappa), (x,v)\in \Omega \times \mathbb{R}^{3},
\end{equation}
where $\Omega :=\frac{1}{\kappa}\Omega'.$ 

\subsection{Sketch of Proof}

Here,  we briefly sketch the idea of the proof of Theorem 1.1. A formal Picard iteration gives the following solution formula for \eqref{integral form}
\begin{equation} \label{Picard}
f=\sum_{i=0}^{\infty}(S_{\Omega}K)^{i}Jg.
\end{equation}
This is a valid solution of \eqref{integral form} if the right hand side of \eqref{Picard} converges in $H^{1}$. 
We first consider the $L^2$ convergence. We have
\begin{lemma}\label{estimate on L2}
For any $h \in L^{2}(\Omega \times \mathbb{R}^{3})$, we have
\begin{equation}
\Vert S_{\Omega}Kh \Vert_{L^{2}(\Omega \times \mathbb{R}^{3})}\lesssim \diam(\Omega)^{\frac{1}{2}} \Vert h \Vert_{L^{2}(\Omega \times \mathbb{R}^{3})}.
\end{equation}
\end{lemma}

By Lemma \ref{estimate on L2} above, we can see that $S_{\Omega}K$ is a contraction mapping in $L^{2}$ provided the diameter of $\Omega$ is small. This provides the $L^{2}$ convergence of \eqref{Picard} in case $\diam(\Omega)$ is small. However, concerning the $H^{1}$ convergence of \eqref{Picard}, we do not have a direct analogy of Lemma \ref{estimate on L2}. Instead, we establish the following key lemmas:

\begin{lemma} \label{estimate on H1} 
Given $h \in H^{1}(\Omega \times \mathbb{R}^{3})$, we have
\begin{equation}
\Vert S_{\Omega}Kh \Vert_{H^{1}(\Omega \times \mathbb{R}^{3})}\lesssim \diam(\Omega)^{\frac{1}{2}}\Vert h \Vert_{H^{1}(\Omega \times \mathbb{R}^{3})}+C(\Omega)\Vert h \Vert_{L^{2}(\partial\Omega \times \mathbb{R}^{3})},
\end{equation}
where $\Vert h \Vert_{L^{2}(\partial\Omega \times \mathbb{R}^{3})}$ is defined in trace sense. 
\end{lemma}

\begin{lemma}\label{BL2est}
Let $C(\Omega)$ be as defined in Lemma \ref{circle lemma}, then for any $h \in L^{2}(\Omega \times \mathbb{R}^{3})$ we have
\begin{equation}
\Vert S_{\Omega}Kh \Vert_{L^{2}(\partial \Omega \times \mathbb{R}^{3})} 
\lesssim C(\Omega)\Vert h \Vert_{L^{2}(\Omega \times \mathbb{R}^{3})}.
\end{equation}
\end{lemma}

Hence, we have
\begin{corollary}
Let $C(\Omega)$ be as defined in Lemma \ref{circle lemma}, then for any $h \in L^{2}(\Omega \times \mathbb{R}^{3})$ we have
\begin{equation}
\Vert S_{\Omega}KS_{\Omega}Kh \Vert_{H^{1}(\Omega\times\mathbb{R}^{3})} 
\lesssim \diam(\Omega)^{\frac{1}{2}}\Vert S_{\Omega}Kh \Vert_{H^{1}(\Omega\times\mathbb{R}^{3})}
+C(\Omega)^{2}\Vert h \Vert_{L^{2}(\Omega \times \mathbb{R}^{3})}.
\end{equation}
\end{corollary}

By Lemma \ref{estimate on H1} we have
\begin{align*}
&\Vert (S_{\Omega}K)^{i}Jg \Vert_{H^{1}(\Omega\times\mathbb{R}^{3})}\\
\lesssim& \diam(\Omega)^{\frac{1}{2}}\Vert (S_{\Omega}K)^{i-1}Jg\Vert_{H^{1}(\Omega\times\mathbb{R}^{3})}+C(\Omega)^{2}\Vert (S_{\Omega}K)^{i-2}Jg \Vert_{L^{2}( \Omega \times \mathbb{R}^{3})}.
\end{align*}
For small $\diam(\Omega)$, we have
\begin{equation} \label{SB}
\begin{split}
&\Vert (S_{\Omega}K)^{i}Jg \Vert_{H^{1}(\Omega\times\mathbb{R}^{3})}\\
\leq& \frac{1}{2}\Vert(S_{\Omega}K)^{i-1}Jg\Vert_{H^{1}(\Omega\times\mathbb{R}^{3})}+CC(\Omega)^{2}\Vert (S_{\Omega}K)^{i-2}Jg \Vert_{L^{2}( \Omega \times \mathbb{R}^{3})}.
\end{split}
\end{equation}
Combining \eqref{SB} with Lemma \ref{estimate on L2}, we conclude the $H^{1}$ convergence of \eqref{Picard} for the case where $\diam(\Omega)$ is small.  Notice that the constant $C(\Omega)$ depends on $\Omega$. More precisely, it depends on the maximum radius of curvature. Nevertheless, it is the diameter of $\Omega$ that affects the convergence, no mater what value $C(\Omega)$ is.   Uniqueness of the $H^1$ solution follows from the contraction mapping argument on the $L^2$ space.

The rest of this article is as follows.  In Section  \ref{sec:bounded_domain}, we present the proofs of  estimates necessary to establish  Theorem \ref{main theorem 1}.  In order the complete these tasks, we need to carefully put the sharp estimates for integrability of $k$, the nature of damped transport, and the geometric properties of the domain into the account.  In the end of Section \ref{sec:bounded_domain} , we complete the proof of the main theorem.  In the 
Section \ref{SufficientC}, we prove the sufficient condition Lemma \ref{lem:H1_Jg}.

\begin{remark}
After this research was finished, the authors notice \cite{CKDecay}  is concerning related issues. In particular, they prove the solutions for time evolutional and stationary full Boltzmann equation are in $W_x^{1,p}$ for $1\leq p<3$ on diffuse reflection boundary problem.

\end{remark}


{\section{Estimates for bootstrap strategy\label{sec:bounded_domain}}
In this section, we first  prove  Lemma \ref{estimate on H1}, Lemma \ref{BL2est}, together with the $L^2$ contraction mapping, Lemma \ref{estimate on L2}.

We  review an important geometric property from \cite{I kun 1, CHK}.}

\begin{lemma} \label{circle lemma}
Let $\Omega$ be a bounded domain with $C^2$ boundary with positive Gaussian curvature. Then, there exists a positive constant $C(\Omega)$ depending on $\Omega$ such that for any $z\in \partial\Omega$ we have 
\begin{equation} \label{E estimate by geo}
\vert z-q(z,v) \vert \lesssim C(\Omega)N(z,v).
\end{equation}
\end{lemma}

{
Here, we investigate some properties on integrability of $k$.   }

\begin{lemma} \label{k int est}
For any $\mu_1 \geq 0$ and $\mu_2 > 0$ with $\mu_1 + \mu_2 < 3$, there exists a positive constant $C(\rho, \mu_1, \mu_2)$ such that 
\begin{equation} \label{k int est 2}
\sup_{v^* \in \R^3} \int_{v \in \mathbb{R}^{3}}\frac{1}{|v|^{\mu_1}}|k(v,v^{*})|^{\mu_2}dv < C(\rho, \mu_1, \mu_2).
\end{equation}
\end{lemma}

\begin{proof}
Using \eqref{AB}, we have

\begin{align*}
 &
 \int_{v \in \mathbb{R}^{3}}\frac{1}{|v|^{\mu_1}}|k(v,v^{*})|^{\mu_2}dv
 \\\lesssim&
  \int_{v \in \mathbb{R}^{3}}\frac{1}{|v|^{\mu_1}}\left(\frac{1}{\vert v-v^{*} \vert(1+\vert v \vert +\vert v^{*} \vert)^{1-\gamma}}e^{-\frac{1-\rho}{4}(\vert v-v^{*} \vert^{2}+(\frac{\vert v \vert^{2}-\vert v^{*} \vert^{2}}{\vert v-v^{*} \vert})^{2})}\right)^{\mu_2}dv
   \\\leq&  \int_{v \in \mathbb{R}^{3}}\frac{1}{|v|^{\mu_1}}\frac{1}{|v-v^{*}|^{\mu_2}}e^{-\frac{(1-\rho)\mu_2}{4}\vert v-v^{*} \vert^{2}}dv
  \\=&
  \int_{|v|\leq |v-v^*|}\frac{1}{|v|^{\mu_1}}\frac{1}{|v-v^{*}|^{\mu_2}}e^{-\frac{(1-\rho)\mu_2}{4}\vert v-v^{*} \vert^{2}}dv\\\ &
  +\int_{|v|\geq |v-v^*|}\frac{1}{|v|^{\mu_1}}\frac{1}{|v-v^{*}|^{\mu_2}}e^{-\frac{(1-\rho)\mu_2}{4}\vert v-v^{*} \vert^{2}}dv
  \\\leq&
  \int_{|v|\leq |v-v^*|}\frac{1}{|v|^{\mu_1+\mu_2}}e^{-\frac{(1-\rho)\mu_2}{4}(|v|^{2})}dv
  +\int_{|v|\geq |v-v^*|}\frac{1}{|v-v^{*}|^{\mu_1+\mu_2}}e^{-\frac{(1-\rho)\mu_2}{4}(\vert v-v^{*} \vert^{2})}dv
  \\\leq&
  \int_{|v|\in \mathbb{R}^{3}}\frac{1}{|v|^{\mu_1+\mu_2}}e^{-\frac{(1-\rho)\mu_2}{4}(|v|^{2})}dv
  +\int_{|v|\in \mathbb{R}^{3}}\frac{1}{|v-v^{*}|^{\mu_1+\mu_2}}e^{-\frac{(1-\rho)\mu_2}{4}(\vert v-v^{*} \vert^{2})}dv
  \\\leq&
  C.
\end{align*}
\end{proof}

We note that, due to the symmetry of the function $k$, we have
\begin{equation} \label{k int est 2*}
\sup_{v \in \R^3} \int_{v^* \in \mathbb{R}^{3}}\frac{1}{|v^*|^{\mu_1}}|k(v,v^{*})|^{\mu_2}dv^* < C(\rho,\mu_1, \mu_2).
\end{equation}
{
To investigate the integrability of derivative of $k$,
we quote an important lemma from Caflisch \cite{Caf 1}. }

\begin{proposition} 
For any $0<\mu < 3$ and $0<\rho<1$   , there exists a positive constant $C(\rho,\mu)$ such that 
\begin{equation} \label{some key poly decay}
 \int_{v^* \in \mathbb{R}^{3}}\frac{1}{|v-v^{*} 
|^\mu}e^{-\frac{1-\rho}{4}(\vert v-v^{*} \vert^{2}+(\frac{\vert v \vert^{2}-\vert v^{*} \vert^{2}}{\vert v-v^{*} \vert})^{2})}dv^* < C(\rho,\mu)(1+|v|)^{-1}.
\end{equation}
\end{proposition}

{ 
Applying \eqref{some key poly decay}, we have the following estimate.

\begin{lemma} \label{k int d est}
For $0 < \mu < \frac32$, there exists a positive constant $C( \mu)$ such that

\begin{equation} \label{k int d est 2*}
\int_{v^* \in \mathbb{R}^{3}}|\nabla_{v}k(v,v^{*})|^{\mu}dv^* < C(\mu)(1+|v|)^{\mu-1}
\end{equation}
\end{lemma}

\begin{proof}
Using the assumption 
\begin{equation} \label{k int d est 2*}
\begin{split}
&\int_{v^* \in \mathbb{R}^{3}}|\nabla_{v}k(v,v^{*})|^{\mu}dv^*\\
&\lesssim
\int_{v^* \in \mathbb{R}^{3}}\left(\frac{1+|v|}{(1+|v|+|v^*|)^{\gamma}|v-v^{*} 
|^2}e^{-\frac{1-\rho}{4}(\vert v-v^{*} \vert^{2}+(\frac{\vert v \vert^{2}-\vert v^{*} \vert^{2}}{\vert v-v^{*} \vert})^{2})}\right)^{\mu}dv^*\\
&\leq 
(1+|v|)^{(1-\gamma)\mu}\int_{v^* \in \mathbb{R}^{3}}\left(\frac{1}{|v-v^{*} 
|^2}e^{-\frac{1-\rho}{4}(\vert v-v^{*} \vert^{2}+(\frac{\vert v \vert^{2}-\vert v^{*} \vert^{2}}{\vert v-v^{*} \vert})^{2})}\right)^{\mu}dv^*\\\end{split}
\end{equation}

\begin{equation} \label{k int d est 2*}
\begin{split}
&\lesssim
(1+|v|)^{(1-\gamma)\mu}(1+|v|)^{-1}\\
&\leq
(1+|v|)^{\mu-1}.
\end{split}
\end{equation}

Here, we use \eqref{some key poly decay} in the third inequality above.

\end{proof}

We introduce a useful change of variables in \cite{I kun 1}, which will be frequently used in the proof of the main theorem. 

\begin{lemma}\label{change of variable lemma}
For nonnegative measurable function $f$, we have

\begin{equation} \label{change of variable lemma eq}
\int_{\mathbb{R}^3}\int_{\Omega}\int_0^{\tau_{x,v}}f(x,v,s) \,dsdxdv=\int_{\mathbb{R}^3}\int_{\Omega}\int_0^{\tau_{y,-u}}f(y+tu,u,t) \,dtdydu
\end{equation}
    
\end{lemma}

\begin{remark}
For the sake of convenience, we put a proof of Lemma \ref{change of variable lemma} in the appendix.
\end{remark}
We define

}\[
\Vert h \Vert_{L^{2}(\partial\Omega\times \mathbb{R}^{3})}^{2}:=\int_{\mathbb{R}^{3}}\int_{\partial \Omega}h^{2}(z,v)d\Sigma(z)dv.
\]

We decompose the proof of Lemma \ref{estimate on H1} into three parts. Recall that
\[
\| S_\O K h \|_{H^1(\O \times \R^3)}^2 = \| S_\O K h \|_{L^2(\O \times \R^3)}^2 + \| \nabla_x S_\O K h \|_{L^2(\O \times \R^3)}^2 + \| \nabla_v S_\O K h \|_{L^2(\O \times \R^3)}^2.
\]
We will give estimates for each term in the right hand side of the above identity.

Notice that the first part is exactly Lemma \ref{estimate on L2}. The proof is as follows.
\begin{proof}[Proof of Lemma \ref{estimate on L2}]
First, by the H\"older inequality, we have
\begin{align*}
&\int_{\mathbb{R}^{3}}\int_{\Omega}\vert S_{\Omega}Kh(x,v) \vert^{2}dxdv
\\=&\int_{\mathbb{R}^{3}}\int_{\Omega}\left| \int_{0}^{\tau_{x,v}} e^{-\nu(v)s}Kh(x-sv,v)ds \right|^{2}dxdv
\\\leq&\int_{\mathbb{R}^{3}}\int_{\Omega}\left( \int_{0}^{\tau_{x,v}} e^{-\nu(v)s}ds \right) \left( \int_{0}^{\tau_{x,v}} e^{-\nu(v)s}\left|Kh(x-sv,v)\right|^{2}ds \right) dxdv
\\\leq&\int_{\mathbb{R}^{3}}\int_{\Omega} \tau_{x, v} \int_{0}^{\tau_{x,v}} e^{-\nu(v)s}\left|Kh(x-sv,v)\right|^{2}ds  dxdv.
\end{align*}
We observe $\tau_{x, v} \leq \diam(\Omega)/|v|$. Hence, we have 
\begin{align*}
 &\int_{\mathbb{R}^{3}}\int_{\Omega}\vert S_{\Omega}Kh(x,v) \vert^{2}dxdv\\ 
 \leq& \diam(\Omega)\int_{\mathbb{R}^{3}}\int_{\Omega} \int_{0}^{\tau_{x,v}} \frac{1}{\vert v \vert}e^{-\nu(v)s} |Kh(x-sv,v)|^{2}ds  dxdv.
\end{align*}

Applying the change of variables in  Lemma \ref{change of variable lemma},  we have
\begin{align*}
&\int_{\mathbb{R}^{3}}\int_{\Omega}  \int_{0}^{\tau_{x,v}} \frac{1}{\vert v \vert}e^{-\nu(v)s}\left|Kh(x-sv,v)\right|^{2} \,ds dx dv
\\=&\int_{\mathbb{R}^{3}} \int_{\Omega} \int_{0}^{\tau_{y,-u}} \frac{1}{\vert u \vert}e^{-\nu(u)t}\left|Kh(y,u)\right|^{2} \,dt dy du
\\=&\int_{\mathbb{R}^{3}}\int_{\Omega}\frac{1}{\vert u \vert} \left|Kh(y,u)\right|^{2} \int_{0}^{\tau_{y,-u}} e^{-\nu(u)t} \,dt dy du
\\=&\int_{\mathbb{R}^{3}}\int_{\Omega}\frac{1}{\vert u \vert} \left|Kh(y,u)\right|^{2}\frac{1}{\nu(u)} \left( 1-e^{\frac{-\nu(u)\vert y-q(x,-u) \vert}{\vert u \vert}} \right) \,dydu
\\\lesssim&\int_{\mathbb{R}^{3}}\int_{\Omega}\frac{1}{\vert u \vert} \left|Kh(y,u)\right|^{2}\,dydu.
\end{align*}
Here, we observe
\begin{align*}
&\int_{\mathbb{R}^{3}}\int_{\Omega}\frac{1}{\vert u \vert} \left| \int_{\mathbb{R}^{3}}k(u,v^{*})h(y,v^{*}) \,dv^{*} \right|^{2} \,dydu
\\\leq&\int_{\mathbb{R}^{3}}\int_{\Omega}\frac{1}{\vert u \vert} \left( \int_{\mathbb{R}^{3}} \vert k(u,v^{*}) \vert \,dv^{*} \right) \left( \int_{R^{3}} \vert k(u,v^{*})  h^{2}(y,v^{*}) \vert \,dv^{*} \right) \,dydu
\\\lesssim&\int_{\mathbb{R}^{3}}\int_{\Omega}\frac{1}{\vert u \vert} \int_{\mathbb{R}^{3}} | k(u,v^{*}) | h^{2}(y,v^{*}) \,dv^{*}dydu
\\=&\int_{\mathbb{R}^{3}}\int_{\Omega} \int_{\mathbb{R}^{3}} \frac{1}{\vert u \vert} | k(u,v^{*})|  h^{2}(y,v^{*}) \,dudydv^{*}
\\=&\int_{\mathbb{R}^{3}}\int_{\Omega} h^{2}(y,v^{*})\int_{\mathbb{R}^{3}} \frac{1}{\vert u \vert}  \vert k(u,v^{*}) \vert \,dudydv^{*}
\\\lesssim&\int_{\mathbb{R}^{3}}\int_{\Omega} h^{2}(y,v^{*}) \,dydv^{*}.
\end{align*}
Notice that, in the above inequalities,  we use \[
\int_{\mathbb{R}^{3}}\left|k(v,v^{*})\right|dv\lesssim 1, \quad \int_{\mathbb{R}^{3}}\frac{1}{\left|v\right|}|k(v,v^{*})|\,dv\lesssim 1,
\] which are conclusions of Lemma  \ref{k int est}.

Summarizing the above estimates, we obtain
\[
\int_{\mathbb{R}^{3}}\int_{\Omega}\vert S_{\Omega}Kh(x,v) \vert^{2}dxdv \lesssim \diam(\Omega) \int_{\mathbb{R}^{3}}\int_{\Omega} h^{2}(x,v)dxdv.
\]
This completes the proof.
\end{proof}

Now, we proceed to prove  Lemma \ref{estimate on H1}.
\begin{proof}[Proof of Lemma \ref{estimate on H1}]
Observe that 
\[
\nabla_{x}S_{\Omega}Kh = S_{\Omega}\nabla_{x}Kh(x,v)+\nabla_{x}\tau_{x,v}e^{-\nu(v)\tau_{x,v}}Kh(q(x,v),v).
\]
From Lemma \ref{estimate on L2}, we have
\[
\int_{\mathbb{R}^{3}}\int_{\Omega}\vert S_{\Omega}K\nabla_{x}h(x,v) \vert^{2}dxdv \lesssim \diam(\Omega) \int_{\mathbb{R}^{3}}\int_{\Omega}\vert \nabla_{x}h(x,v) \vert^{2}dxdv.
\]
Thus, we focus on the estimate for the second term. 
It is known in \cite{GuoKim} that,
\begin{equation} \label{tau_dx}
\nabla_{x}\tau_{x,v}=\frac{-n(x,v)}{N(x,v)\vert v \vert}.
\end{equation}
Notice that for fix $v$,  we can parametrize $\Omega$ by the part of the boundary $\Gamma^{-}_v$ and the travel distance from it in $v$ direction, i.e.,
performing the change of variables $x=z+s\frac{v}{|v|}$ with $z \in \Gamma^{-}_{v}$ and $0 < s <|z- q(z,-v)|$. Notice that the Jacobian of the change of variables is $N(z,v)$, which works in favor with our goal. We have
\begin{align*}
&\int_{\mathbb{R}^{3}}\int_{\Omega} \vert \nabla_x \tau_{x, v} e^{-\nu(v)\tau_{x,v}}Kh(q(x,v),v) \vert^{2}dxdv
\\=&\int_{\mathbb{R}^{3}}\int_{\Omega} \frac{1}{N^{2}(x,v)\vert v \vert^{2}}e^{-2\nu(v)\tau_{x,v}}\vert Kh(q(x,v),v) \vert^{2}dxdv
\\=&\int_{\mathbb{R}^{3}}\int_{\partial\Omega}\int_{0}^{\vert z-q(z,-v) \vert} \frac{1}{N^{2}(z,v)\vert v \vert^{2}}e^{-2\nu(v)\frac{s}{\vert v \vert}}\vert Kh(z,v) \vert^{2}N(z,v)dsd\Sigma(z)dv
\\=&\int_{\mathbb{R}^{3}}\frac{1}{\vert v \vert^{2}}\int_{\partial\Omega}\frac{1}{N(z,v)}\vert Kh(z,v) \vert^{2}\int_{0}^{\vert z-q(z,-v) \vert} e^{-2\nu(v)\frac{s}{\vert v \vert}}dsd\Sigma(z)dv
\\\leq&\int_{\mathbb{R}^{3}}\frac{1}{\vert v \vert^{2}}\int_{\partial\Omega}\frac{1}{N(z,v)}\vert Kh(z,v) \vert^{2} \vert z-q(z,-v) \vert d\Sigma(z)dv
\end{align*}

Using \eqref{E estimate by geo} we get
\begin{align*}
&\int_{\mathbb{R}^{3}}\frac{1}{\vert v \vert^{2}}\int_{\partial\Omega}\frac{1}{N(z,v)}\vert Kh(z,v) \vert^{2}\vert z-q(z,-v) \vert d\Sigma(z)dv
\\\lesssim&C(\Omega)\int_{\mathbb{R}^{3}}\frac{1}{\vert v \vert^{2}}\int_{\partial\Omega}\vert Kh(z,v) \vert^{2} d\Sigma(z)dv.
\end{align*}

In order to absorb the singularity, we carefully distribute the weight of $|k|$ in the H\"older inequality so that
\begin{equation} \label{est:distribution}
\begin{split}
&\int_{\mathbb{R}^{3}}\frac{1}{\vert v \vert^{2}}\int_{\partial\Omega}\vert Kh(z,v) \vert^{2} \,d\Sigma(z)dv\\
\leq&\int_{\mathbb{R}^{3}}\frac{1}{\vert v \vert^{2}}\int_{\partial\Omega} \left( \int_{\mathbb{R}^{3}}|k(v,v^{*})|^{\frac{3}{2}} \,dv^{*} \right) \left( \int_{\mathbb{R}^{3}}|k(v,v^{*})|^{\frac{1}{2}} h^{2}(z,v^{*}) \,dv^{*} \right)\,d\Sigma(z)dv\\
\lesssim&\int_{\mathbb{R}^{3}}\int_{\partial\Omega}h^{2}(z,v^{*}) \left(\int_{\mathbb{R}^{3}}\frac{1}{\vert v \vert^{2}}|k(v,v^{*})|^{\frac{1}{2}} \,dv \right) \,d\Sigma(z)dv^{*}\\
\lesssim&\int_{\mathbb{R}^{3}}\int_{\partial\Omega}h^{2}(z,v^{*})\,d\Sigma(z)dv^{*}.
\end{split}
\end{equation}
Therefore, we have
\[
\int_{\mathbb{R}^{3}}\int_{\Omega} \vert \nabla_x \tau_{x, v} e^{-\nu(v)\tau_{x,v}}Kh(q(x,v),v) \vert^{2}dxdv \lesssim C(\Omega) \int_{\mathbb{R}^{3}}\int_{\partial\Omega}h^{2}(z,v^{*})d\Sigma(z)dv^{*}.
\]
This completes the estimates for $x$ derivatives.  We proceed to deal with $v$ derivatives and prove

\[
\| \nabla_v S_\O K h \|_{L^2(\O \times \R^3)} \lesssim \diam(\O)^\frac{1}{2} \| h \|_{H^1(\O \times \R^3)} + C(\O) \| h \|_{L^2(\partial \O \times \R^3)}.
\]
By a direct computation, we obtain 
\begin{align*}
&\nabla_{v}S_{\Omega}Kh
\\=&\nabla_{v} \int_{0}^{\tau_{x,v}} e^{-\nu(v)s}Kh(x-sv,v)ds
\\=&(\nabla_{v}\tau_{x,v})e^{-\nu(v)\tau_{x,v}}Kh(q(x,v),v)-\int_{0}^{\tau_{x,v}} e^{-\nu(v)s}s\nabla_{v}\nu(v)Kh(x-sv,v)ds
\\&+\int_{0}^{\tau_{x,v}}e^{-\nu(v)s} \int_{\mathbb{R}^{3}}\nabla_{v}k(v,v^{*})h(x-sv,v^{*})dv^{*}ds\\
&-\int_{0}^{\tau_{x,v}}se^{-\nu(v)s} \int_{\mathbb{R}^{3}}k(v,v^{*})\nabla_x h(x-sv,v^{*})dv^{*}ds
\\=&:I_1-I_2+I_3-I_4.
\end{align*}
Notice that, from \cite{GuoKim}, we have 
\begin{equation} \label{tau_dv}
\vert \nabla_{v}\tau_{x,v} \vert \leq \frac{\vert x-q(x,v) \vert \vert  n(q(x,v)) \vert }{\vert v \vert^{2}N(x,v)}= \frac{\vert x-q(x,v) \vert  }{\vert v \vert^{2}N(x,v)}.
\end{equation}
Therefore, we have 
\begin{align*}
\| I_1 \|_{L^2(\O \times \R^3)}^2 = &\int_{\mathbb{R}^{3}}\int_{\Omega}\vert (\nabla_{v}\tau_{x,v})e^{-\nu(v)\tau_{x,v}}Kh(q(x,v),v) \vert^{2}dxdv
\\\leq&\int_{\mathbb{R}^{3}}\int_{\Omega} \frac{\vert x-q(x,v) \vert^{2}}{N^{2}(x,v)\vert v \vert^{4}}e^{-2\nu(v)\tau_{x,v}}\vert Kh(q(x,v),v) \vert^{2}dxdv.
\end{align*}
By performing the change of variable $x=z+s\frac{v}{|v|}$ with $z \in \Gamma^{-}_{v}$, we have
\begin{align*}
&\int_{\mathbb{R}^{3}}\int_{\Omega} \frac{\vert x-q(x,v) \vert^{2}}{N^{2}(x,v)\vert v \vert^{4}}e^{-2\nu(v)\tau_{x,v}}\vert Kh(q(x,v),v) \vert^{2}dxdv
\\=&\int_{\mathbb{R}^{3}}\int_{\partial\Omega}\int_{0}^{\vert z-q(z,-v) \vert} \frac{s^{2}}{N^{2}(z,v)\vert v \vert^{4}}e^{-2\nu(v)\frac{s}{\vert v \vert}}\vert Kh(z,v) \vert^{2}N(z,v)dsd\Sigma(z)dv
\\=&\int_{\mathbb{R}^{3}}\frac{1}{\vert v \vert^{4}}\int_{\partial\Omega}\frac{1}{N(z,v)}\vert Kh(z,v) \vert^{2}\int_{0}^{\vert z-q(z,-v) \vert} s^{2} e^{-2\nu(v)\frac{s}{\vert v \vert}}dsd\Sigma(z)dv.
\end{align*}
By direct calculation, we have 
\begin{equation} \label{est:int_es2}
\begin{split}
&\int_0^{\vert z-q(z,-v) \vert} s^2 e^{-2\nu(v)\frac{s}{\vert v \vert}}ds
\\=&\frac{\vert v \vert^{2}}{2\nu(v)^{2}} \int_0^{\vert z-q(z,-v) \vert} e^{-2\nu(v)\frac{s}{\vert v \vert}}ds
\\&- \left( \frac{\vert v \vert}{2\nu(v)}\vert z-q(z,-v) \vert ^{2}+\frac{\vert v \vert ^{2}}{2\nu(v)^{2}}\vert z-q(z,-v) \vert 
 \right)e^{-\frac{2\nu(v)\vert z-q(z,-v) \vert}{\vert v \vert}}
 \\\lesssim& |v|^2 \vert z-q(z,-v) \vert.
\end{split}
\end{equation}
From  \eqref{est:int_es2}, Lemma \ref{circle lemma}, and Lemma \ref{k int est} ,we have
\begin{equation}
\begin{split}\label{estI1withs}
&\int_{\mathbb{R}^{3}}\frac{1}{\vert v \vert^{4}}\int_{\partial\Omega}\frac{1}{N(z,v)}\vert Kh(z,v) \vert^{2}\int_{0}^{\vert z-q(z,-v) \vert} s^{2} e^{-2\nu(v)\frac{s}{\vert v \vert}}dsd\Sigma(z)dv
\\\lesssim&\int_{\mathbb{R}^{3}}\frac{1}{\vert v \vert ^{2}}\int_{\partial\Omega}\frac{1}{N(z,v)}\vert Kh(z,v) \vert^{2}\vert z-q(z,-v) \vert d\Sigma(z)dv
\\\lesssim&C(\Omega)\int_{\mathbb{R}^{3}}\frac{1}{\vert v \vert ^{2}}\int_{\partial\Omega}\vert Kh(z,v) \vert^{2} d\Sigma(z)dv
\\\leq&C(\Omega)\int_{\mathbb{R}^{3}}\frac{1}{\vert v \vert^{2}}\int_{\partial\Omega}\int_{R^{3}}|k(v,v^{*})|^{\frac{3}{2}} dv^{*}\int_{R^{3}}|k(v,v^{*})|^{\frac{1}{2}} h^{2}(z,v^{*}) dv^{*}d\Sigma(z)dv
\\\lesssim&C(\Omega)\int_{\mathbb{R}^{3}}\int_{\partial\Omega}h^
{2}(z,v^{*})\int_{R^{3}}\frac{1}{\vert v \vert^{2}}|k(v,v^{*})|^{\frac{1}{2}} dvd\Sigma(z)dv^{*}
\\\lesssim&C(\Omega)\int_{\mathbb{R}^{3}}\int_{\partial\Omega}h^
{2}(z,v^{*})d\Sigma(z)dv^{*}.
\end{split}\end{equation}
Thus, we have
\[
\| I_1 \|_{L^2(\O \times \R^3)}^2 \lesssim C(\Omega) \int_{\mathbb{R}^{3}} \int_{\partial\Omega} h^
{2}(z,v^{*})d\Sigma(z)dv^{*}.
\]
Notice that the property  of damped transport equation \eqref{est:int_es2}  reduces the power of $|v|$ in the denominator in \eqref{estI1withs} so that it is integrable.
For $I_2$, we have 
\begin{align*}
&||I_2||_{L^{2}(\O \times \R^3)}^{2} 
\\=&\int_{\mathbb{R}^{3}}\int_{\Omega} \left| \int_{0}^{\tau_{x,v}} e^{-\nu(v)s}s\nabla_{v}\nu(v)Kh(x-sv,v)ds \right| ^{2}dxdv
\\\leq&\int_{\mathbb{R}^{3}}\int_{\Omega} \left( \int_{0}^{\tau_{x,v}} s e^{-\nu(v)s} \,ds \right) \left( \int_{0}^{\tau_{x,v}} s e^{-\nu(v)s} \left| \nabla_{v}\nu(v)Kh(x-sv,v)\right| ^{2} \,ds \right) \,dxdv.
\end{align*}
On  the one hand, 
\begin{equation} 
\int_{0}^{\tau_{x,v}} s e^{-\nu(v)s} ds=\frac1{\nu(v)^2}\int_0^{\nu(v)\tau_{x,v}}ze^{-z}dz\leq \frac1{\nu(v)^2} .
\end{equation}
By direct calculation,
\begin{equation} \label{est:int_es}
\int_{0}^{\tau_{x,v}} s e^{-\nu(v)s} ds = - \frac{\tau_{x, v}}{\nu(v)} e^{-\nu(v) \tau_{x, v}}+ \frac{1}{\nu(v)} \int_0^{\tau_{x, v}} e^{-\nu(v)s}\,ds \lesssim \tau_{x, v}.
\end{equation}
Using Lemma \ref{change of variable lemma}, we have 
\begin{align*}
&\int_{\mathbb{R}^{3}}\int_{\Omega} \frac{1}{\nu(v)^{2}} \int_{0}^{\tau_{x,v}}   e^{-\nu(v)s}s\vert \nabla_{v}\nu(v)Kh(x-sv,v)\vert ^{2}ds dxdv
\\= &\int_{\mathbb{R}^{3}}\int_{\Omega} \frac{1}{\nu(u)^{2}} \int_{0}^{\tau_{y,-u}}   e^{-\nu(u)t}t\vert \nabla_{u}\nu(u)Kh(y,u)\vert ^{2}dt dydu
\\= &\int_{\mathbb{R}^{3}}\int_{\Omega} \frac{1}{\nu(u)^{2}}\vert \nabla_{u}\nu(u)Kh(y,u)\vert ^{2} \int_{0}^{\tau_{y,-u}}   e^{-\nu(u)t}tdt dydu
\\\lesssim &\int_{\mathbb{R}^{3}}\int_{\Omega} \frac{1}{\nu(u)^{2}}\vert \nabla_{u}\nu(u)Kh(y,u)\vert ^{2} \tau_{y, -u} dydu.
\\\lesssim &\int_{\mathbb{R}^{3}}\int_{\Omega}(1+\vert u \vert)^{ -2}\vert Kh(y,u)\vert ^{2}\tau_{y,-u}  dydu
\\ \lesssim &\int_{\mathbb{R}^{3}}\int_{\Omega}\vert Kh(y,u)\vert ^{2}\frac{\vert y-q(y,-u)\vert }{\vert u \vert } dydu.
\end{align*}
Notice we use \eqref{AA} and \eqref{AD} in Assumption A in inequalities above. Using H\"{o}lder's inequality again,  
\begin{align*}
&\int_{\mathbb{R}^{3}}\int_{\Omega}\vert Kh(y,u)\vert ^{2}\frac{\vert y-q(y,-u)\vert }{\vert u \vert } \,dydu
\\\lesssim & \diam(\Omega)\int_{\mathbb{R}^{3}}\int_{\Omega}\frac{1}{\vert u \vert}\left| \int_{\mathbb{R}^{3}}h(y,v^{*})k(u,v^{*})dv^{*}\right| ^{2}  \,dydu
\\\leq &\diam(\Omega)\int_{\mathbb{R}^{3}}\int_{\Omega}\frac{1}{\vert u \vert} \left(\int_{\mathbb{R}^{3}}\vert k(u,v^{*}) \vert \,dv^{*} \right) \left(\int_{\mathbb{R}^{3}}h^{2}(y,v^{*}) \vert k(u,v^{*}) \vert \,dv^{*} \right) \,dydu
\\\lesssim &\diam(\Omega)\int_{\mathbb{R}^{3}}\int_{\Omega}\int_{\mathbb{R}^{3}}\frac{1}{\vert u \vert}h^{2}(y,v^{*}) \vert k(u,v^{*}) \vert \,dv^{*}dydu
\\= &\diam(\Omega)\int_{\mathbb{R}^{3}}\int_{\Omega}h^{2}(y,v^{*}) \left( \int_{\mathbb{R}^{3}} \frac{1}{\vert u \vert}\vert k(u,v^{*}) \vert \,du \right) \,dydv^{*}
\\\lesssim &\diam(\Omega)\int_{\mathbb{R}^{3}}\int_{\Omega}h^{2}(y,v^{*}) \,dydv^{*}.
\end{align*}
Thus, we have
\[
\| I_2 \|_{L^2(\O \times \R^3)}^2 \lesssim \diam(\Omega)\int_{\mathbb{R}^{3}}\int_{\Omega}h^{2}(y,v^{*}) \,dydv^{*}.
\]
For $I_3$, in the same way as for $I_1$ and $I_2$, we obtain
\begin{align*}
&||I_3||^{2}_{L^{2}(\O \times \R^3)}
\\=&\int_{\mathbb{R}^{3}}\int_{\Omega} \left| \int_{0}^{\tau_{x,v}}e^{-\nu(v)s} \int_{\mathbb{R}^{3}}\nabla_{v}k(v,v^{*})h(x-sv,v^{*})dv^{*}ds \right| ^{2} dxdv
\\\lesssim&\int_{\mathbb{R}^{3}}\int_{\Omega} \tau_{x, v} \int_{0}^{\tau_{x,v}} e^{-\nu(v)s}\left| \int_{\mathbb{R}^{3}}\nabla_{v}k(v,v^{*})h(x-sv,v^{*})dv^{*} \right| ^{2}ds  dxdv
\\\lesssim &\int_{\mathbb{R}^{3}}\int_{\Omega} \tau_{x, v}  \int_{0}^{\tau_{x,-v}} e^{-\nu(v)s}\left| \int_{\mathbb{R}^{3}}\nabla_{v}k(v,v^{*})h(x,v^{*})dv^{*} \right| ^{2}ds dxdv
\\\lesssim &\diam(\Omega) \int_{\mathbb{R}^{3}}\int_{\Omega}\frac{1}{\vert v \vert} \int_{0}^{\tau_{x,-v}} e^{-\nu(v)s}\left| \int_{\mathbb{R}^{3}}\nabla_{v}k(v,v^{*})h(x,v^{*})dv^{*} \right| ^{2}ds  dxdv.
\end{align*}
Again we perform the H\"older inequality and using Lemma \ref{k int d est} :
\begin{align*}
&\int_{\mathbb{R}^{3}}\int_{\Omega}\frac{1}{\vert v \vert} \int_{0}^{\tau_{x,-v}} e^{-\nu(v)s}\left| \int_{\mathbb{R}^{3}}\nabla_{v}k(v,v^{*})h(x,v^{*})dv^{*} \right| ^{2} \,dsdxdv
\\\lesssim &\int_{\mathbb{R}^{3}}\int_{\Omega}\frac{1}{\vert v \vert} \left| \int_{\mathbb{R}^{3}}\nabla_{v}k(v,v^{*})h(x,v^{*})dv^{*} \right| ^{2} \,dxdv
\\\leq&\int_{\mathbb{R}^{3}}\int_{\Omega}\frac{1}{\vert v \vert} \left(\int_{\mathbb{R}^{3}} \vert \nabla_{v}k(v,v^{*}) \vert^\frac{5}{4} \,dv^{*} \right) \left( \int_{\mathbb{R}^{3}} \vert \nabla_{v}k(v,v^{*}) \vert^\frac{3}{4} h^{2}(x,v^{*}) \,dv^{*} \right) \,dxdv
\\\lesssim&\int_{\mathbb{R}^{3}}\int_{\Omega}\frac{(1+|v|)^{\frac{1}{4}}}{\vert v \vert} \left( \int_{\mathbb{R}^{3}} \vert \nabla_{v}k(v,v^{*}) \vert^\frac{3}{4} h^{2}(x,v^{*}) \,dv^{*} \right) \,dxdv
\\=&\int_{\mathbb{R}^{3}}\int_{\Omega} h^{2}(x,v^{*}) \left(\int_{\mathbb{R}^{3}} \frac{(1+|v|)^{\frac{1}{4}}}{\vert v \vert}  \vert \nabla_{v}k(v,v^{*}) \vert^\frac{3}{4} \,dv \right) \,dxdv^{*}
\\\lesssim&\int_{\mathbb{R}^{3}}\int_{\Omega} h^{2}(x,v^{*}) \left( \int_{\mathbb{R}^{3}} \frac{(1+|v|)^{\frac{1}{4}}}{|v|} \frac{(1+|v|)^{\frac{3}{4}}}{|v-v^*|^{\frac{3}{2}}}e^{-\frac{3}{16}(1-\rho)|v-v^*|^2} \,dv \right) \,dxdv^{*}
\\=&\int_{\mathbb{R}^{3}}\int_{\Omega} h^{2}(x,v^{*}) \left( \int_{\mathbb{R}^{3}} \frac{(1+|v|)}{|v|} \frac{1}{|v-v^*|^{\frac{3}{2}}}e^{-\frac{3}{16}(1-\rho)|v-v^*|^2}  \,dv \right)\,dxdv^{*}
\\\lesssim&\int_{\mathbb{R}^{3}}\int_{\Omega} h^{2}(x,v^{*}) \,dxdv^{*}.
\end{align*}

Finally,  we proceed to the estimate for $I_4$. 
\begin{align*}
&\| I_4 \|_{L^2(\O \times \R^3)}^2
\\=&\int_{\mathbb{R}^{3}}\int_{\Omega} \left| \int_{0}^{\tau_{x,v}}se^{-\nu(v)s} \int_{\mathbb{R}^{3}}k(v,v^{*})\nabla_x h(x-sv,v^{*})dv^{*}ds \right|^{2}dxdv
\\\leq&\int_{\mathbb{R}^{3}}\int_{\Omega} \left( \int_{0}^{\tau_{x,v}} se^{-\nu(v)s}ds \right)\\ 
&\quad \times \left( \int_{0}^{\tau_{x,v}} se^{-\nu(v)s} \left( \int_{\mathbb{R}^{3}}\vert k(v,v^{*})\nabla_x h(x-sv,v^{*})\vert dv^{*} \right)^2 ds \right) dxdv.
\end{align*}
From \eqref{est:int_es}, we have
\[
 \int_{0}^{\tau_{x,v}} se^{-\nu(v)s}\,ds \lesssim \tau_{x, v} \leq \frac{\diam(\O)}{|v|}.
\]
We also have
\[
\int_0^{\tau_{x, v}} se^{-\nu(v)s}\,ds \leq \frac{1}{\nu_0^2} \int_0^\infty z e^{-z}\,dz \lesssim 1.
\]
Substituting these into the original integral, changing the variable and using the H\"older inequality again, we have
\begin{align*}
&\int_{\mathbb{R}^{3}}\int_{\Omega} \left( \int_{0}^{\tau_{x,v}} se^{-\nu(v)s}ds \right)\\ 
&\quad \times \left( \int_{0}^{\tau_{x,v}} se^{-\nu(v)s} \left( \int_{\mathbb{R}^{3}}\vert k(v,v^{*})\nabla_x h(x-sv,v^{*})\vert dv^{*} \right)^2 ds \right) dxdv
\\\lesssim&\int_{\mathbb{R}^{3}}\int_{\Omega} \frac{\diam(\Omega)}{|v|}\int_{0}^{\tau_{x,v}} se^{-\nu(v)s} \left( \int_{\mathbb{R}^{3}}\vert k(v,v^{*})\nabla_x h(x-sv,v^{*})\vert dv^{*} \right)^2 ds  dxdv
\\\lesssim&\int_{\mathbb{R}^{3}}\int_{\Omega} \frac{\diam(\Omega)}{|v|}\int_{0}^{\tau_{x,-v}} se^{-\nu(v)s}\int_{\mathbb{R}^{3}} |k(v,v^{*})|dv^{*}\int_{\mathbb{R}^{3}} |k(v,v^{*})||\nabla_x h(x,v^{*})|^2 dv^{*}ds  dxdv
\\\lesssim&\int_{\mathbb{R}^{3}} \int_{\Omega} \frac{\diam(\Omega)}{|v|}\int_{\mathbb{R}^{3}}|k(v,v^{*})||\nabla_x h(x,v^{*})|^2 dv^{*}  dxdv
\\\lesssim& \diam(\Omega)\int_{\mathbb{R}^{3}}\int_{\Omega} |\nabla_x h(x,v^{*})|^2 dxdv^{*}.
\end{align*}

Therefore, Lemma \ref{estimate on H1} is proved.
\end{proof}
Now, we present  the proof of Lemma \ref{BL2est}.
 \begin{proof}[Proof of Lemma \ref{BL2est}]

 Similar to the definition of $\Gamma_v^-$ in Lemma \ref{lem:H1_Jg}, we can define
 
 \begin{align}
 \Gamma_{v}^{0} &:= \lbrace x \in \partial \Omega \mid n(x)\cdot v =0 \rbrace,\\
 \Gamma_{v}^{+} &:= \lbrace x \in \partial \Omega \mid n(x)\cdot v >0 \rbrace.
 \end{align}

For fixed $v$, $\partial \Omega$ can be written as the disjoint union of $\Gamma_v^-$, $\Gamma_v^0$, and $\Gamma_v^+$. Since $\Omega$ satisfies Assumption A, $\Gamma_v^0$ is measure zero. Notice that for $z\in \Gamma_v^-$, $S_\Omega Kh(z,v)=0$. On the other hand, when $z \in \Gamma_v^+$, we have
\begin{equation}
\begin{split}
S_\Omega K h (z,v)&=\int_0^{\tau_{z,v}
}e^{-\nu(|v|)s}K(f)(z-sv,v)ds \\&=\int_0^{|z-q(z,v)|}\int_{\mathbb{R}^3}  \frac1{|v|}e^{-\frac{\nu(|v|)}{|v|}r}k(v,v^*)h(z-r\hat{v},v^*)dv^*dr,
\end{split}
\end{equation}
where $\hat{v}=v/|v|$.

Therefore, \begin{equation}\begin{split}
&\Vert S_\Omega Kh\Vert ^2_{L^2(\partial\Omega\times\mathbb{R}^3)}=\int_{\mathbb{R}^3}\int_{\partial\Omega}|S_\Omega K h(z,v)|^2d\Sigma(z)dv\\&=\int_{\mathbb{R}^3}\int_{\Gamma_v^+}\left|\int_0^{|z-q(z,v)|}\int_{\mathbb{R}^3}  \frac1{|v|}e^{-\frac{\nu(|v|)}{|v|}r}k(v,v^*)h(z-r\hat{v},v^*)dv^*dr\right|^2d \Sigma(z)dv\\&\leq \int_{\mathbb{R}^3}\int_{\Gamma_v^+}\left(\int_0^{|z-q(z,v)|}\int_{\mathbb{R}^3}  e^{-2\frac{\nu(|v|)}{|v|}r}|k(v,v^*)|^{\frac32}dv^*dr\right)\\&\left(\int_0^{|z-q(z,v)|}\int_{\mathbb{R}^3}\frac{|k(v,v^*)|^{\frac12}}{|v|^2}|h(z-r\hat{v},v^*)|^2dv^*dr\right)d \Sigma(z)dv\\&\lesssim \int_{\mathbb{R}^3}\int_{\Gamma_v^+}|z-q(z,v)|\int_0^{|z-q(z,v)|}\int_{\mathbb{R}^3}\frac{|k(v,v^*)|^{\frac12}}{|v|^2}|h(z-r\hat{v},v^*)|^2dv^*drd \Sigma(z)dv\\&\lesssim C(\Omega)\int_{\mathbb{R}^3}\int_{\Gamma_v^+}\int_0^{|z-q(z,v)|}\int_{\mathbb{R}^3}\frac{|k(v,v^*)|^{\frac12}}{|v|^2}|h(z-r\hat{v},v^*)|^2N(z,v)dv^*drd \Sigma(z)dv.
\end{split}\end{equation}
Letting $x=z-r\hat{v}$, we have
\begin{equation}\begin{split}
\Vert S_\Omega Kh\Vert ^2_{L^2(\partial\Omega\times\mathbb{R}^3)} &\lesssim C(\Omega)\int_{\mathbb{R}^3}\int_{\Omega}\int_{\mathbb{R}^3}\frac{|k(v,v^*)|^{\frac12}}{|v|^2}|h(x,v^*)|^2dv^*dxdv\\&\lesssim C(\Omega)\int_{\Omega}\int_{\mathbb{R}^3}|h(x,v^*)|^2dv^*dx.
\end{split}\end{equation}
\end{proof} 

Finally,  we present the remained part of the proof of Theorem \ref{main theorem 1}. 
\begin{proof}[Poof of Theorem \ref{main theorem 1}]
We recall \eqref{SB} that for $i > 1$, we have
\[
|| (S_{\Omega}K)^{i}Jg ||_{H^{1}(\Omega\times\mathbb{R}^{3})}-\frac{1}{2}|| (S_{\Omega}K)^{i-1}Jg||_{H^{1}(\Omega\times\mathbb{R}^{3})}\leq C C(\Omega)^{2}||(S_{\Omega}K)^{i-2}Jg ||_{L^{2}( \Omega \times \mathbb{R}^{3})}.
\]

Summing up the inequality from $i=2$ to $n$, because of cancellation of terms,  we have

\begin{equation}
\begin{split}
&-\frac12 \Vert S_{\Omega}KJg \Vert_{H^{1}(\Omega\times\mathbb{R}^{3})}+\frac{1}{2}\sum_{i=2}^{n-1}\Vert (S_{\Omega}K)^{i}Jg \Vert_{H^{1}(\Omega\times\mathbb{R}^{3})} +\Vert (S_{\Omega}K)^{n}Jg \Vert_{H^{1}(\Omega\times\mathbb{R}^{3})}     \\&
 \leq C\sum_{i=2}^{n}C(\Omega)^{2}\Vert (S_{\Omega}K)^{i-2}Jg \Vert_{L^{2}( \Omega \times \mathbb{R}^{3})}.  
\end{split}  
\end{equation}
Therefore, \begin{equation}
\begin{split}
&\frac{1}{2}\sum_{i=0}^{n-1}\Vert (S_{\Omega}K)^{i}Jg \Vert_{H^{1}(\Omega\times\mathbb{R}^{3})}  \\&\leq \frac12\Vert Jg \Vert_{H^{1}(\Omega\times\mathbb{R}^{3})}    
 + C\sum_{i=2}^{n}C(\Omega)^{2}\Vert (S_{\Omega}K)^{i-2}Jg \Vert_{L^{2}( \Omega \times \mathbb{R}^{3})}.  
\end{split}  
\end{equation}

From Lemma \ref{estimate on L2},  we know the  right hand side convergences when diameter of $\Omega$ is small enough as $n$ tends to infinity, and, hence, so does the left hand side. Therefore, we conclude the sufficient part of the Theorem  \ref{main theorem 1}.

On the other hand, if $f$ belongs to $H^{1}(\Omega\times\mathbb{R}^{3})$, then $S_{\Omega}Kf$ also belongs to $H^{1}(\Omega\times\mathbb{R}^{3})$. Thus,  $Jg=f-S_{\Omega}Kf$ also belongs to $H^{1}(\Omega\times\mathbb{R}^{3})$.

\end{proof}

\section{Sufficient condition on boundary data\label{SufficientC}}

 In Section 1, we see that Lemma \ref{lem:H1_Jg} gives a explicit condition on the function on $\Gamma_-$  for Theorem \ref{main theorem 1} to hold.. In this Section, we shall prove Lemma \ref{lem:H1_Jg} by proving Lemma \ref{est_Jg_L2}, Lemma \ref{est_Jg_H1_x} and Lemma \ref{est_Jg_H1_v}.

We introduce a notation: First we regard $\partial\Omega$ as a manifold with an embedding function $\psi:M \to \mathbb{R}^{3}$ in $\mathbb{R}^{3}$. Given $x\in \partial\Omega$, a neighborhood $U_{x}$ of $x$, a point $p \in U_{x}$ and a chart function $\phi_{x}:U_{x} \to \mathbb{R}^{2}$, we define the metric $g$ as 
\begin{equation}
g_{ij}(p):=<X_{i}(p),X_{j}(p)>_{\mathbb{R}^{3}},    
\end{equation} where 
\begin{equation}
X_{i}(p):=\frac{\partial}{\partial x_{i}}(\psi \circ \phi_{x}^{-1})|_{\phi_{x}(p)}.    
\end{equation}
We define  the Euclidean norm of the covariant gradient of $h$ at $x$ on $\partial\Omega$ as
\[
|\nabla_{x}h|^{2}:=g^{ij}\frac{\partial}{\partial x_{i}}(h\circ \phi_{x}^{-1})\frac{\partial}{\partial x_{j}}(h\circ \phi_{x}^{-1}).
\]
Notice that this definition is independent of choice of chart. 

\begin{lemma} \label{est_Jg_L2}
For $g$ satisfying the assumptions in Lemma \ref{lem:H1_Jg}, we have
\[
\| Jg \|_{L^2(\O \times \R^3)} \leq \diam(\O)^\frac{1}{2} \| g \|_{L^2(\Gamma^-)}.
\]
\end{lemma}

\begin{proof}
Let $x=z+s\frac{v}{|v|}$, where $z \in \Gamma^{-}_{v}$. Then, we have
\begin{align*}
&\int_{\mathbb{R}^{3}}\int_{\Omega}e^{-2\nu(v)\tau_{x,v}}|g(q(x,v),v)|^{2}dxdv
\\=&\int_{\mathbb{R}^{3}}\int_{\Gamma_{v}^{-}}\int_{0}^{|z-q(x,-v)|}e^{-2\nu(v)\frac{s}{|v|}}|g(z, v)|^{2}N(z,v)dsd\Sigma(z)dv
\\=&\int_{\mathbb{R}^{3}}\int_{\Gamma_{v}^{-}}|g(z, v)|^{2}\int_{0}^{|z-q(z,-v)|}e^{-2\nu(v)\frac{s}{|v|}}dsd\Sigma(z)dv
\\\leq& \diam(\Omega)\int_{\mathbb{R}^{3}}\int_{\Gamma_{v}^{-}}|g(z,v)|^{2}d\Sigma(z)dv.
\end{align*}
This completes the proof.
\end{proof}

\begin{lemma} \label{est_Jg_H1_x}
For $g$ satisfying the assumptions in Lemma \ref{lem:H1_Jg}, we have
\begin{align*}
\| \nabla_x Jg \|_{L^2(\O \times \R^3)}^2 \lesssim& \int_{\R^3} \frac{\nu(v)^2}{|v|^2} \int_{\Gamma^-_v} |g(z, v)|^2\,d\Sigma(z) dv\\ 
&+ \int_{\R^3} \int_{\Gamma^-_v} |\nabla_x g (z, v)|^2\,d\Sigma(z) dv.
\end{align*}
\end{lemma}

\begin{proof}
Recalling \eqref{tau_dx}, we get
\begin{align*}
&\nabla_{x} Jg(x, v) 
\\=&(\nabla_{x}e^{-\nu(v)\tau_{x,v}})g(q(x,v),v)+e^{-\nu(s)\tau_{x,v}}\nabla_{x}(g(q(x,v),v))
\\=&\nu(v)e^{-\nu(v)\tau_{x,v}}\frac{n(x,v)}{N(x,v)\vert v \vert}  g(q(x,v),v)+e^{-\nu(s)\tau_{x,v}}  \nabla_{x}(g(q(x,v),v)).
\end{align*}
Hence, we have
\begin{align*}
&\vert \nabla_{x} Jg(x, v) \vert ^{2}
\\\lesssim&\nu(v)^{2}e^{-2\nu(v)\tau_{x,v}}\frac{1}{N^{2}(x,v)\vert v \vert^{2}}\vert g(q(x,v),v)\vert^{2}+e^{-2\nu(s)\tau_{x,v}} \vert \nabla_{x}(g(q(x,v),v))\vert^{2}
\\\lesssim&\nu(v)^{2}e^{-2\nu(v)\tau_{x,v}}\frac{1}{N^{2}(x,v)\vert v \vert^{2}}\vert g(q(x,v),v)\vert^{2}+e^{-2\nu(s)\tau_{x,v}}\vert \nabla_{x}g(q(x,v),v)\vert^{2}\vert \nabla_{x}q(x,v)\vert^{2}
\\=&\nu(v)^{2}e^{-2\nu(v)\tau_{x,v}}\frac{1}{N^{2}(x,v)\vert v \vert^{2}}\vert g(q(x,v),v)\vert^{2}+e^{-2\nu(s)\tau_{x,v}}\vert \nabla_{x}g(q(x,v),v)\vert^{2}\frac{1}{N^{2}(x,v)}
\\= &I_5+I_6.
\end{align*}
Here, we used the fact from \cite{GuoKim} that 

\begin{equation}
 \vert \nabla_{x}q(x,v)\vert=\frac{1}{N(x,v)}.   
\end{equation} By changing the coordinates again as we do in the proof of Theorem \ref{main theorem 1}, we see that the integral of $I_5$ is bounded.
For $I_6$, we also use change of variable by letting $y=q'(x,v)$:
\begin{align*}
&\int_{\mathbb{R}^{3}} \int_{\Omega} e^{-2\nu(v)\tau_{x,v}}\vert \nabla_{x}g(q(x,v),v)\vert^{2}\frac{1}{N^{2}(x,v)}dxdv
\\\lesssim& C(\Omega)\int_{\mathbb{R}^{3}}\int_{\Gamma_{v}^{-}}\vert \nabla_{x}g(y,v)\vert^{2}d\Sigma(y)dv,
\end{align*}
which is finite.
\end{proof}

\begin{lemma} \label{est_Jg_H1_v}
For $g$ satisfying the assumptions in Lemma \ref{lem:H1_Jg}, we have
\begin{align*}
\| \nabla_v Jg \|_{L^2(\O \times \R^3)}^2 \lesssim& \int_{\R^3} \int_{\Gamma^-_v} |g(z, v)|^2\,d\Sigma(z) dv\\ 
&+ \int_{\R^3} \frac{\nu(v)^2}{|v|^2} \int_{\Gamma^-_v} |g(z, v)|^2\,d\Sigma(z) dv\\
&+ \int_{\R^3} \int_{\Gamma^-_v} |\nabla_x g(z, v)|^2\,d\Sigma(z) dv\\ 
&+ \int_{\R^3} \int_{\Gamma^-_v} |\nabla_v g(z, v)|^2 N^2(z, v)\,d\Sigma(z) dv.
\end{align*}
\end{lemma}

\begin{proof}
From \eqref{tau_dv}, we have
\begin{equation} \label{random term}
\begin{split} 
&\vert \nabla_{v}e^{-\nu(v)\tau_{x,v}}g(q(x,v),v) \vert ^{2}
\\\lesssim &\vert(1+\vert v \vert )^{\gamma -1}\tau_{x,v}e^{-\nu(v)\tau_{x,v}}g(q(x,v),v)\vert^{2}\\
&+ \nu^{2}(v)\frac{\vert x-q(x,v) \vert^{2} }{\vert v \vert^{4}N^{2}(x,v)}e^{-2\nu(v)\tau_{x,v}} \vert g(q(x,v),v) \vert^{2}
\\&+ e^{-2\nu(s)\tau_{x,v}}\vert \nabla_{x}g(q(x,v),v)\vert^{2} \frac{\vert x-q(x,v) \vert ^{2} }{\vert v \vert ^{2} N(x,v)^{2}}  +\vert e^{-\nu(s)\tau_{x,v}}(\nabla_{v}g)(q(x,v),v)\vert^{2}.
\end{split}
\end{equation}
In what follows, we estimate the integral of each term in the right hand side of \eqref{random term}.

For the first term, we use change of variable. Let $x=z+sv/|v|$, where $z \in \Gamma^{-}_{v}$, so we have
\begin{align*}
& \int_{\mathbb{R}^{3}} \int_{\Omega} \vert (1+\vert v \vert )^{\gamma -1}\tau_{x,v}e^{-\nu(v)\tau_{x,v}}g(q(x,v),v) \vert^{2} dxdv
\\=& \int_{\mathbb{R}^{3}}\int_{\Omega}(1+\vert v \vert )^{2\gamma -2}\tau_{x,v}^{2}e^{-2\nu(v)\tau_{x,v}}\vert g(q(x,v),v)\vert^{2}dxdv
\\=& \int_{\mathbb{R}^{3}}\int_{\Gamma_v^-}\int_{0}^{\vert z-q(z,-v) \vert}(1+\vert v \vert )^{2\gamma -2}\frac{s^{2}}{\vert v   \vert^{2}}e^{-2\nu(v)\frac{s}{\vert v   \vert}}\vert g(z,v)\vert^{2}N(z,v)dsd\Sigma(z)dv
\\=& \int_{\mathbb{R}^{3}}\frac{1}{\vert v   \vert^{2}}(1+\vert v \vert )^{2\gamma -2}\int_{\Gamma_v^-}N(z,v)\vert g(z,v)\vert^{2}\int_{0}^{\vert z-q(z,-v) \vert}s^{2}e^{-2\nu(v)\frac{s}{\vert v   \vert}}dsd\Sigma(z)dv.
\end{align*}
From \eqref{est:int_es2}, it suffices to show that the integral
\[
\int_{\mathbb{R}^{3}}\int_{\Gamma_v^-}\vert g(z,v)\vert^{2}(1+\vert v \vert )^{2\gamma -2}N(z,v)|z - q(z, -v)|d\Sigma(z)dv
\]
converges.

By Lemma \ref{circle lemma} we have
\begin{align*}
& \int_{\mathbb{R}^{3}}\int_{\Gamma_v^-}\vert g(z,v)\vert^{2}(1+\vert v \vert )^{2\gamma-2}N(z,v)\vert z-q(z,-v) \vert d\Sigma(z)dv
\\\lesssim& \int_{\mathbb{R}^{3}}\int_{\Gamma_v^-}\vert g(z,v)\vert^{2} N^{2}(z,v) d\Sigma(z)dv
\\\lesssim& \int_{\mathbb{R}^{3}}\int_{\Gamma_v^-}\vert g(z,v)\vert^{2}d\Sigma(z)dv.\end{align*}

For the second term, once we use the exact same change of variable as we did before and calculate the exactly same integral we have
\begin{align*} 
& \int_{\mathbb{R}^{3}} \int_{\Omega} (\nu^{2}(v)\frac{\vert x-q(x,v) \vert^{2}  }{\vert v \vert^{4}N^{2}(x,v)}e^{-2\nu(v)\tau_{x,v}}) \vert g(q(x,v),v) \vert^{2} dxdv
\\\lesssim& \int_{\mathbb{R}^{3}}\int_{\Gamma_v^-}\vert g(z,v)\vert^{2}\frac{\nu^{2}(v)}{N(x,v)\vert v \vert ^{4}}\frac{\vert v \vert ^{2}}{\nu(v)^{2}}|z-q(z,-v)|  d\Sigma(z)dv
\\\lesssim& C(\Omega)\int_{\mathbb{R}^{3}}\int_{\Gamma_v^-}\vert g(z,v)\vert^{2}\frac{1}{\vert v \vert ^{2}}d\Sigma(z)dv.
\end{align*}

For the third term, we have
\begin{align*} 
& \int_{\mathbb{R}^{3}} \int_{\Omega} e^{-2\nu(s)\tau_{x,v}}\vert \nabla_{x}g(q(x,v),v)\vert^{2} \frac{\vert x-q(x,v) \vert ^{2} }{\vert v \vert ^{2} N(x,v)^{2}} dxdv
\\\lesssim& \int_{\mathbb{R}^{3}}\int_{\Gamma_v^-}\vert \nabla_{x}g(z,v)\vert^{2}\frac{1}{N(x,v)\vert v \vert ^{2}}\frac{\vert v \vert ^{2}}{\nu(v)^{2}}|z-q(z,-v)|  d\Sigma(z)dv
\\\lesssim& C(\Omega)\int_{\mathbb{R}^{3}}\int_{\Gamma_v^-}\vert \nabla_{x}g(z,v)\vert^{2}  d\Sigma(z)dv
\\<&\infty.
\end{align*}

For the last term, notice that
\begin{align*} 
& \int_{\mathbb{R}^{3}} \int_{\Omega} \vert e^{-\nu(s)\tau_{x,v}}(\nabla_{v}g)(q(x,v),v)\vert^{2} dxdv
\\\lesssim& \int_{\mathbb{R}^{3}}\int_{\Gamma_v^-}\vert \nabla_{v}g(z,v)\vert^{2}N(x,v)|z-q(z,-v)|  d\Sigma(z)dv
\\\lesssim& \int_{\mathbb{R}^{3}}\int_{\Gamma_v^-}\vert \nabla_{v}g(z,v)\vert^{2}N^2(x,v)  d\Sigma(z)dv
\\<&\infty.
\end{align*}

This completes the proof.


\end{proof}

Here we present a simple sufficient condition on $g$ so that the regularity of $Jg$ would meet the assumption of Theorem \ref{main theorem 1}: 

\begin{corollary}
Suppose for given $x,y \in \partial \Omega$ we have 
\begin{align*}
g(x,v)\lesssim& e^{-a\vert v \vert^{2}},\\
\vert g(x,v)-g(y,v)\vert \lesssim& \vert x-y \vert e^{-a\vert v \vert^{2}}.
\end{align*}
Then we have $Jg \in H^{1}(\Omega \times \mathbb{R}^{3})$.    
\end{corollary} 

\section{Appendix}

In this appendix, we give a proof to Lemma \ref{change of variable lemma}, see also \cite{I kun 1}.
\begin{lemma} 
For non-negative measurable function $f$, we have

\[ 
 \int_{\mathbb{R}^3}\int_{\Omega}\int_0^{\tau_{x,v}}f(x,v,s)dsdxdv=\int_{\mathbb{R}^3}\int_{\Omega}\int_0^{\tau_{y,-u}}f(y+tu,u,t)dtdydu.
\]
    
\end{lemma}

\begin{proof}

Consider the following sets:
\begin{align*}
&A:=\{(v,x,s) \mid v\in\R^3,x\in \Omega, 0\leq s \leq \tau_{x,v}\},\\
&B:=\{(u,y,t) \mid u\in\R^3,y\in \Omega, 0\leq t \leq \tau_{y,-u}\}.
\end{align*}
Also, Consider the change of variable:
\begin{align*}
    \sigma_1:A \longrightarrow B,& \quad  (\sigma_1(v),\sigma_1(x),\sigma_1(s)):= (v,x-vs,s),\\
    \sigma_2:B \longrightarrow A,& \quad (\sigma_2(u),\sigma_2(y),\sigma_2(t)):= (u,y+ut,t).
\end{align*}
Notice that the for any $(v,x,s) \in A$ we have $x-vs \in \Omega$, $\tau_{\sigma_1(x),-\sigma_1(v)}=\tau_{x-sv,-v}=\tau_{x,-v}+s \geq s =\sigma_1(s) \geq 0$. Hence the map $\sigma_1$ is well defined. Similarly, the map $\sigma_2$ is also well defined, and the fact $\sigma_1 \circ \sigma_2=I_B,\sigma_2 \circ \sigma_1=I_A$, show that the map is a bijection. This with the identities of Jacobean implies \eqref{change of variable lemma eq}.
    
\end{proof}

\begin{center}
\includegraphics{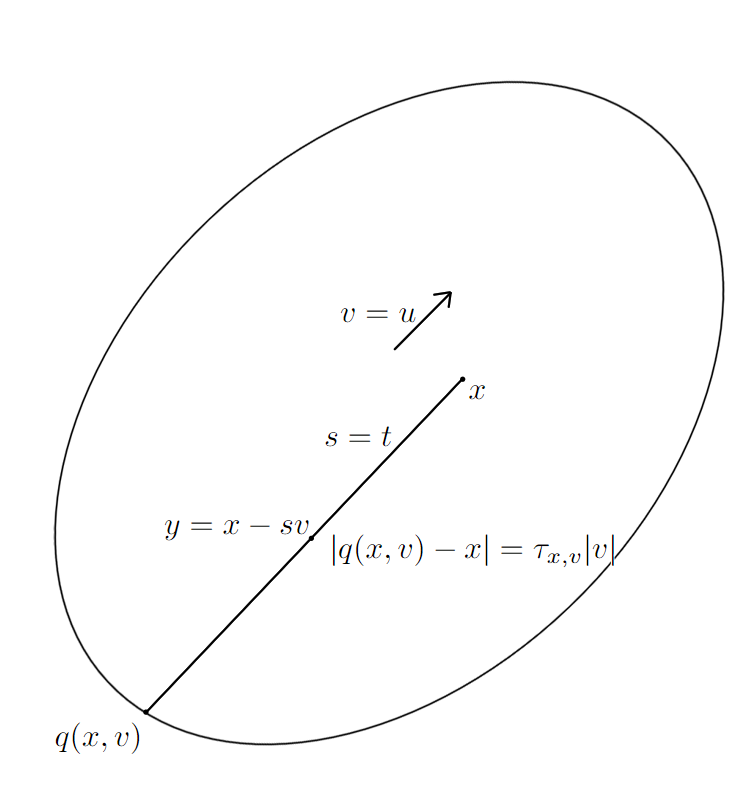}
\end{center}

\section*{Acknowledgement}

This publication is supported in part by the National Taiwan University-Kyoto University Joint Funding. The forth author is supported by JSPS KAKENHI grant number 20K14344.


\begin{thebibliography}{9}








\bibitem{Caf 1} Caflisch, R.~E.
The Boltzmann equation with a soft potential. I. Linear, spatially-homogeneous.
\textit{Comm. Math. Phys.} 74 (1980), no. 1, 71--95.

    
\bibitem{Ces} Cessenat, M. 
Th\'eor\`emes de trace $L^p$ pour des espaces de fonctions de la neutronique. (French. English summary) [Trace theorems in $L^p$ for neutronic function spaces]
\textit{C. R. Acad. Sci. Paris Sér.} I Math. 299 (1984), no. 16, 831--834.

\bibitem{CCLS} Chen, C-C.; Chen, I-K.; Liu, T-P.; Sone, Y. Thermal transpiration for the linearized Boltzmann equation \textit{Comm. Pure Appl. Math.} 60 (2007), 147--163. 

\bibitem{IKC} Chen, I-K. Boundary singularity of moments for the  linearized Boltzmann equation  \textit{J. Stat. Phys.} 153 (2013), no. 1, 93--118. 

\bibitem{RegularChen} Chen, I-K. Regularity of stationary solutions to the linearized Boltzmann equations, \textit{SIAM J. Math. Anal.} 50 (2018), no. 1, 138--161. 

\bibitem{chenH} Chen, H. Regularity of Boltzmann equation with Cercignani-Lampis boundary in convex domain. \textit{arXiv preprint arXiv:2105.10052.} (2021).

\bibitem{I kun 1}
Chen, I-K.; Chuang, P-H.; Hsia, C-H.; Su, J-K.;
A revisit of the velocity averaging lemma: on the regularity of stationary Boltzmann equation in a bounded convex domain.
\textit{J. Stat. Phys.} 189 (2022), no. 17, 43pp.

\bibitem{CFLT} Chen, I-K.; Funagane, H.; Takata, S.; Liu, T.-P. Singularity of the velocity distribution function in molecular velocity space, \textit{Comm. Math. Phys.} 341 (2016), no. 1, 105--134. 

\bibitem{ChenHsia} Chen, I-K.; Hsia, C-H. Singularity of macroscopic variables near boundary for gases with cutoff hard potential. \textit{SIAM J. Math. Anal.} 47 (2015), no. 6, 4332--4349.

\bibitem{CHK} Chen, I-K.; Hsia, C-H.; Kawagoe, D.: Regularity for diffuse reflection boundary problem to the stationary linearized Boltzmann equation in a convex domain, \textit{Ann. I.~H.~Poincar\'e} 36 (2019) 745--782.

\bibitem{chenkim} Chen, H.; Kim, C. Regularity of stationary Boltzmann equation in convex domains. \textit{Arch. Ration. Mech. Anal. } 244 (2022), no. 3, 1099--1222.


\bibitem{CKDecay} Chen, H.; Kim, C. Gradient Decay in the Boltzmann theory of Non-isothermal boundary \textit{	arXiv:2304.06933}.

\bibitem{CLT} Chen, I-K.; Liu, T-P.; Takata, S.:  Boundary singularity for thermal transpiration problem of the linearized Boltzmann equation. \textit{ Arch. Ration. Mech. Anal.} 212 (2014), no. 2, 575--595.


        
          

\bibitem{GuoKim}Esposito, R.; Guo, Y.; Kim, C.; Marra, R. Non-isothermal boundary in the Boltzmann theory and Fourier law.\textit{Comm. Math. Phys. }323 (2013), no. 1, 177--239. 



\bibitem{Grad} Grad, H. Asymptotic theory of the Boltzmann equation, II \textit{1963 Rarefied Gas Dynamics }(Proc. 3rd Internat. Sympos., Palais de l'UNESCO, Paris, 1962), Vol. I pp. 26--59 Academic Press, New York. 

\bibitem{Grad 2} Grad, H.
High frequency sound according to the Boltzmann equation.
\textit{SIAM J. Appl. Math.} 14 (1966), 935--955.

\bibitem{Gui 1} Guiraud, J.~P. Probl\`{e}me aux limites int\'erieur pour l'\'equation de Boltzmann lin\'eaire. 
\textit{J. M\'ecanique} 9 (1970), 443--490.

\bibitem{Guiraudlinear}Guiraud, J.P.: Probleme aux limites int\'erieur pour l'\'equation de Boltzmann lin\'eaire.\textit{ J. de M\'ec.} 9 (1970), No.3, 183--231. 

\bibitem{Gui 2} Probl\`eme aux limites int\'erieures pour l'\'equation de Boltzmann. (French) Actes du Congr\`es International des Math\'ematicien
 (Nice, 1970), Tome 3, . Gauthier-Villars, Paris, (1971), 115--122.

\bibitem{Guiraudnonlinear}Guiraud, J.P.: Probleme aux limites int\'erieur pour l'\'equation de Boltzmann en r\'egime stationnaire, faiblement non lin\'eaire. \textit{J. de M\'ec. }11(2) (1972), 443--490.

\bibitem{Guo} Guo, Y.: Decay and continuity of the Boltzmann equation in bounded domains, \textit{Arch. Ration. Mech. Anal.}, 197 (2010), 713--809.
 
\bibitem{GKTT} Guo,Y.; Kim, C.; Tonon, D; Trescases, A. Regularity of the Boltzmann Equation in convex domains.  arXiv:1212.1694 (\textit{Inventiones Mathematicae}, online first, 2016)

\bibitem{kawagoe 1} Kawagoe, D. Regularity of solutions to the stationary transport equation with the incoming boundary data, Ph.D. thesis (2018).




















































\end{thebibliography}
\end{document}